\documentclass[12pt]{amsart}
\usepackage{amssymb,amsmath,amscd,amsthm}
\usepackage{amsfonts,epsfig,latexsym,graphicx,amssymb}%,pict2e}
\usepackage[T1]{fontenc}

\usepackage{hyperref}

\newtheorem{Lm}{Lemma}
\newtheorem{Thm}[Lm]{Theorem}
\newtheorem{Prop}[Lm]{Proposition}

\newtheorem{Cor}{Corollary}

\newtheorem{Add}{Addendum}

\theoremstyle{definition}
\newtheorem{Def}{Definition}
\newtheorem{Rem}[Lm]{Remark}
\newtheorem{Ex}{Example}

\def\bS{\mathbb{S}}
\def\E{\mathbb{E}}
\def\F{\mathbb{F}}

\def\R{\mathbb{R}}

\def\Z{\mathbb{Z}}
\def\N{\mathbb{N}}
\def\T{\mathbb{T}}
\def\P{\mathbb{P}}

\def\eps{\varepsilon}

\def\tJ{\widetilde{J}}

\def\mF{\mathcal{F}}

\def\mL{\mathcal{L}}
\def\mA{\mathcal{A}}
\def\mR{\mathcal{R}}
\def\tmA{\widetilde{\mathcal{A}}}
\def\tmR{\widetilde{\mathcal{R}}}

\def\Homeo{\mathop{\mathrm{Homeo}}}
\def\HS{\Homeo_+(\bS^1)}
\def\diam{\mathop{\mathrm{diam}}}
\def\FN{F_{N}}

\def\bdef{\begin{Def}}
\def\endef{\end{Def}}
\def\bthm{\begin{Thm}}
\def\ethm{\end{Thm}}
\def\bprop{\begin{Prop}}
\def\enprop{\end{Prop}}
\def\blm{\begin{Lm}}
\def\elm{\end{Lm}}
\def\bcor{\begin{Cor}}
\def\ecor{\end{Cor}}
\def\brm{\begin{Rem}}
\def\erm{\end{Rem}}
\def\bfig{\begin{picture}}
\def\efig{\end{picture}}
\def\beq{\begin{eqnarray}}
\def\eneq{\end{eqnarray}}
\def\beal{\begin{aligned}}
\def\enal{\end{aligned}}

\newcommand{\tf}{\tilde{f}}
\newcommand{\tg}{\tilde{g}}

\title{Translation numbers define generators of $F_k^+\to {\text{\rm Homeo}_+}(\bS^1)$}

\author[T. Golenishcheva--Kutuzova]{Tatiana Golenishcheva--Kutuzova}

\address{Moscow Center for Continuous Mathematical Education, Moscow, Russia}

\email{tania@mccme.ru}

\thanks{T. G.--K. was supported in part by RFBR grant 13-01-00969-a and joint RFBR/CNRS grant 10-01-93115-CNRS\_a.}

\author[A. Gorodetski]{Anton Gorodetski}

\address{Department of Mathematics, University of California, Irvine CA 92697, USA}

\email{asgor@math.uci.edu}

\thanks{A.\ G.\ was supported in part by NSF grant IIS-1018433.} % DMS--0901627 and

\author[V. Kleptsyn]{Victor Kleptsyn}

\address{CNRS, Institute of Mathematical Research of Rennes (IRMAR, UMR 6625 du CNRS), France}

\email{victor.kleptsyn@univ-rennes1.fr}

\thanks{V. K. was supported in part by RFBR grant 13-01-00969-a and joint RFBR/CNRS grant 10-01-93115-CNRS\_a.}

\author[D. Volk]{Denis Volk}

\address{KTH Matematik, Lindstedsv\"agen 25, SE-100 44 Stockholm Sweden}

\email{dvolk@kth.se}

\thanks{D.\ V.\ was supported in part by grants RFBR 12-01-31241-mol\_a, RFBR 13-01-00969-a, President's of Russia MK-7567.2013.1}

\date{\today}

\begin{document}

\begin{abstract}
We consider a minimal action of a finitely generated semigroup by homeomorphisms of a circle, and show that the collection of translation numbers of individual elements completely determines the set of generators (up to a common continuous change of coordinates). One of the main tools used in the proof is the synchronization properties of random dynamics of circle homeomorphisms: Antonov's theorem and its corollaries.
\end{abstract}

\maketitle

%\tableofcontents

\section{Introduction}

Groups of homeomorphisms (diffeomorphisms) of a circle  form a classical object of research (see~\cite{Ghys, Ghys-circle, Navas, FS}) and still attract lots of attention~\cite{Calegari-Walker, Navas-actions, Ghys-Sergiescu, DKN}. These groups are related to the questions in geometry~\cite{Matsumoto, W}, dynamics~\cite{GI1, GI2}, algebra~\cite{Mo}, and other fields.

    In the case of $\mathbb{Z}$ orientation-preserving action, i.e. of dynamics of iterates of one orientation-preserving circle homeomorphism, the most important dynamical invariant is the \emph{rotation number} introduced by Poincare~\cite{P}. For minimal actions, this number is a full invariant: any two such actions with the same rotation number are conjugate. Recall that due to the Denjoy Theorem~\cite{D} the irrationality of a rotation number of sufficiently smooth action implies minimality; moreover, even without the assumption of smoothness the action with an irrational rotation number is semi-conjugate to the corresponding rotation.
Rotation number is defined up to an integer, and sometimes it is convenient to fix a lift of the circle map. In this case we will consider a corresponding \emph{translation number} which is a real number uniquely determined by the lift. Specifically, if $f:S^1\to S^1$ is a homeomorphism of a circle, we will denote by $\tilde f$ a lift of $f$, and the translation number $\tau(\tf):= \lim_{n\to\infty} (\tf^n(x)-x)/n$ is known to exist and to be independent of~$x\in S^1$.

      Rotation (translation) numbers play an important role also in the case of the group action on a circle. In particular, for the case of minimal \emph{group} actions on the circle, a theorem by Ghys (see~\cite{Ghys},~\cite[Theorem 6.5]{Ghys-circle}) says that an action $\phi$ of a group $G$ is uniquely determined (up to a conjugacy) by its bounded Euler class $\phi^*(eu)\in H^2_b(G,Z)$. Applying this to a free group $G=F_k$, one can see that this class corresponds to the translation numbers of all the compositions, defined up to a simultaneous change of lifts. This result was further used by Matsumoto to show that for a surface group its smooth action with the highest possible Euler class is unique up to the conjugacy, see~\cite{Matsumoto}.

 Here we relax the conditions of Ghys' Theorem and show that a similar statement hold for a finitely generated semigroup action as well. In other words, it is enough to know the translation numbers of all {\it positive} finite compositions of generators to determine generators (up to a continuous change of coordinates) of the semigroup.

\begin{Thm}\label{t:main} Let $\{f_1, f_2, \ldots, f_k\}$ be orientation preserving homeomorphisms of a circle such that the semigroup $\langle f_1, f_2, \ldots, f_k\rangle_{+}$ generated by them acts minimally on $S^1$, and the same holds for the simigroup generated by their inverses. The collection $\{f_1, \ldots, f_k\}$ with these properties is determined uniquely, up to a simultaneous change of coordinates, by the translation numbers of  positive compositions of their lifts.
\end{Thm}

The minimality assumption cannot be omitted as the following example shows.

\begin{Ex}
Let $I_a, I_r\subset S^1$ be two disjoint closed intervals on a circle, and $f_i, i=1, 2, \ldots, k, $ be orientation preserving homeomorphisms of a circle exhibiting exactly two fixed points, one attracting in $I_a$, and another repelling in $I_r$. Then any positive composition must have zero rotation number and, moreover, one can choose lifts $\tilde{f}_1, \ldots, \tilde{f}_k$ in such a way that any positive composition of the lifts has zero translation number. On the other hand, a simultaneous continuous change of coordinates on $S^1$ must preserve the order of attractors (repellers) of $f_1, \ldots, f_k$. This gives combinatorial invariants for  $\{f_1, \ldots, f_k\}$ that cannot be detected by translation numbers.
\end{Ex}

 The following example shows that one cannot consider the rotation numbers instead of the translation numbers in Theorem~\ref{t:main}.

\begin{Ex}
The group $PSL_2(\Z)$ acts minimally on the circle in the standard way: by projectivization of the corresponding linear maps. This group has a subgroup $\Gamma_2$, generated by $\left[\left(\begin{smallmatrix} 1 & 2 \\ 0 & 1 \end{smallmatrix} \right) \right]$ and $\left[\left(\begin{smallmatrix} 1 & 0 \\ 2 & 1 \end{smallmatrix} \right) \right]$. The subgroup $\Gamma_2$ is a free group, as it can be seen from the ping-pong lemma (see, e.g.,~\cite{Ghys-circle},~\cite[Sec.~5.2]{DKN}). Also, $\Gamma_2$ is an index~$6$ subgroup in $PSL_2(\Z)$, as it is the kernel of the reduction mod~$2$ map from $PSL_2(\Z)$ to $PSL_2(\F_2)=Sym_3$. Hence, as a finite index subgroup in a group acting minimally, $\Gamma_2$ also acts minimally on the circle.

On the other hand, all the elements of $\Gamma_2$ are either hyperbolic, or parabolic: elliptic elements of $PSL_2(\Z)$ are of finite order, while in a free group there are no finite order elements. Hence, all these elements have zero rotation number. Thus, taking any set of elements that generate $\Gamma_2$ as a semigroup, we obtain an action from $\mF$ with all the rotation numbers that are equal to zero. Though, it is clear that such sets are not necessarily conjugate (for instance, the identity may belong to one of them and not to the other).
\end{Ex}

One of the motivations for Theorem~\ref{t:main} comes from recent results by Calegari---Walker~\cite{Calegari-Walker}, where the following question was considered. Given rotation numbers of generators, what can be said about possible rotation number of a specific element of a group (or a semigroup)? As due to Ghys' theorem the translation numbers of all the elements of the group define the action up to a conjugacy, Calegari and Walker have argued that the set of such families of translation numbers of compositions can be thought as a ``character variety'' for the associated non-linear action on the circle. In their work, the main focus was made (and the strongest results were obtained) in the case of positive products of generators.

Another source of motivation comes from the question on dynamics of minimal iterated function systems (that can be modeled using step skew products). In the case of actions on a circle some examples were constructed in~\cite{GI1, GI2}, and later different kinds of generalizations (in particular, to the case of a higher dimensional phase space) appeared in~\cite{GIKN, GHS, GS, HN, V}. In this case it is natural to consider actions of semigroups instead of group actions.

Though the result of Theorem~\ref{t:main} is purely deterministic, one of the key ingredients of the proof is the application of results describing \emph{random} dynamics (together with the arguments ``if an event has a positive probability, it is non-empty''). The setting of random dynamics is the following: assume that we are given a probability measure $\nu$ on the finite set of indices $\{1,\dots,k\}$. To this measure, one can associate a sequence of random iterations $F[n,\omega]=f_{\omega_n}\circ\dots \circ f_{\omega_1}$, where $\omega\in\{1,\dots,k\}^{\N}$ is chosen with respect to the Bernoulli measure $\overline{\nu}:=\nu^{\N}$ and hence, the maps applied on different steps are chosen independently. Also, it is convenient to consider the transformation~$T_\nu$, acting on the space of measures on the circle: $T_\nu (\mu)=\sum_i \nu(i) (f_i)_* \mu$.

For the random dynamics on the circle, the following fundamental result establishes the phenomenon of synchronization under random dynamics:
\begin{Thm}[Antonov, \cite{A}]\label{t:contraction}
Let maps $f_1,\dots,f_k\in \HS$  satisfy the assumptions of Theorem~\ref{t:main}, and let $\nu$ be measure supported on~$\{1,\dots,k\}$. Then exactly one of the following statements holds:
\begin{enumerate}
\item There exists a common invariant measure of all the $f_i$'s, and all these maps are simultaneously conjugate to rotations;
\item\label{i:contr} For any two points $x,y\in\bS^1$ the distance between their random images $F[n,\omega](x)$ and $F[n,\omega](y)$ tends to zero almost surely;
\item There exists $l>1$ and an order $l$ orientation-preserving homeomorphism $\varphi$, such that it commutes with all the $f_i$'s, and after passing to the quotient circle $\bS^1/(\varphi^i(x)\sim \varphi^j(x))$ for the new maps $g_i$ the conclusions of case~(\ref{i:contr}) are satisfied.
\end{enumerate}
In the case of type~(\ref{i:contr}) dynamics, for any point $p\in \bS^1$ the distribution of its random images after $n$ iterations converges to the unique stationary measure~$\mu_+$:
$$
T_\nu^n \delta_p \rightarrow \mu_+ \text{ as } n\to \infty,
$$
where $\delta_p$ stays for the Dirac measure concentrated at~$p$, and $\mu_+=T_\nu \mu_+$.
\end{Thm}

In fact, we need a more detailed description of the random dynamics. In particular, we need the following

\begin{Add}\label{p:12}
In the setting of Theorem~\ref{t:contraction}, assume that the case~(\ref{i:contr}) holds. Then almost surely there exists a point $r=r(\omega)\in\bS^1$ such that for any its neighborhood $V$ the images of its complement are contracted:
$$
\diam F[n,\omega](S^1\backslash V)\to 0\quad \text{as } n\to \infty.
$$
Moreover, $r$ is distributed w.r.t. measure $\mu_-$, which is the unique stationary measure for the dynamics of inverse maps, that is, $\mu_-=\sum_j \nu(j) (f_j^{-1})_* \mu_-$.
\end{Add}

We also notice that the rotation numbers allow us to distinguish between the three cases in Theorem~\ref{t:contraction}. Namely, we have the following

\begin{Add}\label{p:detector}
Given $N\in \mathbb{N}$, consider distribution of rotation numbers of random compositions of $f_1, \ldots, f_k$ with random length $n\in \{1, \ldots, N\}$, $\mathbb{P}(n=i)=\frac{1}{N}, i=1, \ldots, N$. We have the following asymptotics as $N\to \infty$:
\begin{enumerate}
\item\label{i:det1} In the case of a common invariant measure (i.e. in the case (1) in Theorem~\ref{t:contraction}), the rotation numbers of compositions are asymptotically equidistributed on the circle $S^1=\R/\Z$;
\item\label{i:det2} In the case when there is no common invariant measure and dynamics is  non-factorizable (the case (2) in Theorem \ref{t:contraction}), the probability that a rotation number is equal to zero tends to~$1$ as $N\to\infty$;
\item\label{i:det3} In the case of $l$-leaves factorizable dynamics (the case (3) in Theorem \ref{t:contraction}), the rotation number is asymptotically equidistributed on the set $\{0,1/l,\dots,(l-1)/l\}$ as $N\to\infty$.
\end{enumerate}
\end{Add}

The proof of both Addendum~\ref{p:12} and Addendum~\ref{p:detector} (as well as a brief reminder of the proof of Antonov's Theorem~\ref{t:contraction}% in the terms of the current paper
)  is provided in Section~\ref{s:random}.

  Notice that existence of a common invariant measure for several homeomorphisms of a circle, as well as existence of a factorization (i.e. of a homeomorphism of finite order that commutes with all the homeomorphisms from the initial collection) is certainly a degeneracy, and therefore Theorem 2 claims, in particular, that generically, in the case of minimal dynamics, for any two initial conditions under the random application of the homeomorphisms $\{f_1, f_2, \ldots, f_k\}$ their images will eventually become indistinguishable. This is a typical case of a {\it non-linear synchronization}. 
  
  The history of synchronization goes back to the 17th century when Christiaan Huygens observed a phenomenon of synchronization
of pendulum clocks~\cite{Hu}. Currently synchronization is a well established notion in physics~\cite{ABVRS,PRK}.  Theorem~\ref{t:contraction} (stated here in a slightly modified way) was suggested as a dynamical model of synchronization by Antonov~\cite{A} and later was independently re-discovered with a different proof in~\cite{KN}.  Similar features of the non-linear random walks on a line were studied in~\cite{DKNP}.  It is interesting to compare these results with ``synchronization'' of orbits on the same central leaf in some partially hyperbolic systems (\cite{RW, SW}, see also~\cite{H2}) and skew products~\cite{KV1, KV2, H1}. Recently some results of this kind were obtained also in higher dimensions~\cite{H3,V}.

\subsection*{Acknowledgments}
All four authors are former students of Professor Yulij Ilyashenko, whose influence and support they would not be able to overestimate. This work is based on ideas and discussions that they had when all of them took part in the work of the seminar in Dynamical Systems under the guidance of Professor Ilyashenko in Moscow State University.

The authors are grateful to \'E. Ghys and D. Calegari for helpful discussions. V.~K. and D.~V. also would like to thank University of California, Irvine, and the organizers of the International Conference Beyond Uniform Hyperbolicity 2013 in B\k{e}dlewo for their hospitality.

\section{Proof of Theorem \ref{t:main}}

Here we prove Theorem~\ref{t:main} using the properties of the random dynamics on a circle -- Theorem~\ref{t:contraction} and Addendum \ref{p:12} and \ref{p:detector}. The proofs of these random dynamics results are provided in Section \ref{s:random}.

Assume that two collections of generators $\{f_1^{(1)},\dots,f_k^{(1)}\}$ and $\{f_1^{(2)},\dots,f_k^{(2)}\}$ together with their lifts $\{\tf_1^{(1)},\dots,\tf_k^{(1)}\}$ and $\{\tf_1^{(2)},\dots,\tf_k^{(2)}\}$ are given. Suppose also that the translation numbers of the corresponding positive compositions coincide, i.e. for any finite word $\omega_1\omega_2\ldots \omega_n$, $\omega_i\in \{1, \ldots, k\}$, we have $\tau(\tf_{\omega_n}^{(1)}\circ\dots \circ \tf_{\omega_1}^{(1)})=\tau(\tf_{\omega_n}^{(2)}\circ\dots \circ \tf_{\omega_1}^{(2)})$.  For brevity, we denote by $\rho^{(1)}$ and $\rho^{(2)}$ the actions of $F_k^+$ on the circle, associated to these collections of generators, and by $\tilde\rho^{(1)}$ and $\tilde\rho^{(2)}$ respectively the actions on $\R$, associated to the aforementioned lifts.

Recall that there are three possible types of dynamics under assumptions of Theorem~\ref{t:main} (due to Antonov's Theorem~\ref{t:contraction}):
\begin{enumerate}
\item\label{i:measure} The dynamics with a common invariant measure;
\item\label{i:generic} Generic case: no common invariant measure, non-factorizable dynamics;
\item\label{i:cover} The dynamics is factorizable, that is, all the maps $f_i$ commute with the same homeomorphism $\varphi\in\HS$ of finite order~$l\ge 2$; on the quotient by the action of~$\varphi$, the dynamics is of the second type: non-factorizable and with no common invariant measure.
\end{enumerate}
We will refer to them as actions of types~(\ref{i:measure}),~(\ref{i:generic}) and~(\ref{i:cover}) respectively. Addedum~\ref{p:detector} allows us to detect which of the three cases of Theorem~\ref{t:contraction} corresponds to the action if translation numbers are given. Hence, if for the lifts $\tilde\rho^{(1)}, \tilde\rho^{(2)}$ their translation numbers coincide, both actions $\rho^{(1)}, \rho^{(2)}$ belong to the same type. Let us consider each type separately.

\subsection*{Type (1): the dynamics with a common invariant measure}

If  the generators $\{f_1, \ldots, f_k\}$ of an action $\rho$ exhibit a common invariant measure $\mu$, the function
$$
x\mapsto \begin{cases}
\tilde\mu([0,x]), & x\ge 0, \\
-\tilde\mu([x,0]), & x<0
\end{cases}
$$
maps the lift $\tilde\mu$ of this measure to the Lebesgue measure on~$\R$, and hence conjugates the dynamics to an action by translations by corresponding translation numbers. Hence, if lifts of two actions of type~(\ref{i:measure}) have the same translation numbers, they are (topologically) conjugate to the same action, and thus are conjugate.

%\end{proof}

\subsection*{Type (2): generic case (no common invariant measure, non-factorizable dynamics)}

Let us first sketch the proof under an additional assumption that the action is by $C^1$-diffeomorphisms; the general case uses similar arguments, but technically is more involved. First, for two such actions, there exists a word $w\in F_k^+$, for which both corresponding maps $\rho^{(1)}(w)$ and $\rho^{(2)}(w)$ are Morse-Smale with one attractor and one repeller. Indeed, due to some contraction-type arguments (sf.~\cite{DKN}),  a long random composition satisfies this property with probability that tends to one as the length of the composition tends to infinity. Proposition~\ref{p:7} is the generalization of this property that we need for the actions by homeomorphisms.

Now, consider a composition $gf^n h$ with large $n$, where $f$ is a Morse-Smale map with one attractor $a_f$ and one repeller $r_f$, and notice that it is also a Morse-Smale map with the attractor close to $g(a_f)$ and the repeller close to $h^{-1}(r_f)$, provided that these points are different (which is a generic situation). See Lemma~\ref{l:lim-h} below for a continuous analogue of these arguments. Finally, we claim that the translation numbers of lifts reflect the cyclic order of attractors and repellers of the compositions: see Lemma~\ref{l:arar} and Lemma~\ref{l:cfgh-h} below as well as Fig.~\ref{f:1}. Combining these arguments, one can conclude (we omit the technical details in this sketch) that the translation numbers reflect the cyclic order of the orbit of the attractor of a Morse-Smale map. This orbit is  dense in the circle, so mapping the orbit for one action to the corresponding orbit for another action (their cyclic order should be the same due to the coincidence of the translation numbers) we obtain the conjugacy between the actions.

Unfortunately, for an action by homeomorphisms the fixed points of compositions are generically non-isolated. To handle this technical difficulty, we replace iterations of a single Morse-Smale map by a convenient sequence of maps:

\begin{Def}
The family $\{\FN\}_{n\in \mathbb{N}}\subset\HS$ is an $(a,r)$-\emph{family}, where $a,r\in \bS^1$, if for any neighborhoods $U,V$ of $a$ and $r$ respectively for all $N$ sufficiently big one has $F_N(\bS^1\setminus V)\subset U$. If $a\neq r$, we say that it is an $(a,r)$-\emph{MS} family.
\end{Def}

\begin{Lm}\label{l:lim-h}
If $\{\FN\}$ be an $(a,r)$-family of homeomorphisms, $g,h\in \HS$, then $\{g\FN h\}$ is $(g(a),h^{-1}(r))$-family. In particular, if $g(a)\neq h^{-1}(r)$, the family $\{g\FN h\}$ is $(g(a),h^{-1}(r))$-MS.
\end{Lm}

%\begin{Lm}\label{l:lim-h}
%Let $\{\FN\}$ be an $(a,r)$-family of homeomorphisms, $g,h\in \HS$, and suppose that $g(a)\neq h^{-1}(r)$. Then, $\{g\FN h\}$ is $(g(a),h^{-1}(r))$-MS family.
%\end{Lm}

\begin{proof}[Proof of Lemma~\ref{l:lim-h}]
Let $U$ and $V$ be neighborhoods of $g(a)$ and $h^{-1}(r)$ respectively; then, $U'=g^{-1}(U)$ and $V'=h(V)$ are neighborhoods of $a$ and $r$ respectively, and due to the definition for all $N$ sufficiently big we have $\FN(\bS^1\setminus V')\subset U'$. For any such $N$, 
$$
(g\FN h)(\bS^1\setminus V) = g(\FN(\bS^1\setminus V'))\subset g(U')=U.
$$
The last statement is a direct application of the definition.
%As $r\in h(V)$, $a\in g^{-1}(U)$, for any sufficiently big~$N$ one has $\FN(\bS^1\setminus h(V))\subset g^{-1}(U)$.
\end{proof}

\begin{Prop}\label{p:7}
Let $\rho^{(j)}$, $j=1,2$, be two actions of type~(\ref{i:generic}). Then there exist $a^{(j)},r^{(j)}\in\bS^1$, $j=1,2$, and a sequence of words $\{w^N\}$  such that the sequence of maps $\FN^{(j)}:=\rho^{(j)}(w^N)$ is $(a^{(j)},r^{(j)})$-MS, $j=1,2$.
\end{Prop}

\begin{proof}[Proof of Proposition \ref{p:7}]
Fix any probability measure $\nu$ supported on $\{1,\dots,k\}$ (for instance, the uniform one). This measure automatically defines the Bernoulli measure $\overline{\nu}=\nu^{\N}$ on the set of infinite sequences $\omega=\omega_1\omega_2\omega_3\ldots $, and hence the random dynamics in the sense of Theorem~\ref{t:contraction}.

Due to Addendum~\ref{p:12}, for almost every sequence $\omega$ there are random repellers $r^{(1)}$ and $r^{(2)}$ such that for any neighborhoods $V_1$ of $r^{(1)}$ and $V_2$ of $r^{(2)}$ we have
$$
\diam F^{(1)}[n,\omega](\bS^1\backslash V_1)\to 0 \ \ \ \text{and}\ \ \ \diam F^{(2)}[n,\omega](\bS^1\backslash V_2)\to 0 \quad \text{as } n\to \infty,
$$
where $F^{(j)}[n,\omega]=f_{\omega_n}^{(j)}\circ\dots \circ f_{\omega_1}^{(j)}, \ j=1,2.$

Take any such sequence $\omega$ and pick any point $(p^{(1)},p^{(2)})\in \bS^1\times \bS^1$ such that $p^{(j)}\neq r^{(j)}, \ j=1, 2$. Consider a sequence of points $P_n=\left(F^{(1)}[n,\omega](p), F^{(2)}[n,\omega](p)\right)\in \bS^1\times \bS^1=\mathbb{T}^2$. By compactness, there exists a subsequence $\{n_N\}_{N\in \mathbb{N}}$ such that $\{P_{n_N}\}$ converges,
$P_{n_N}\to\left(a^{(1)}, a^{(2)}\right)\in \bS^1\times \bS^1$ as $N\to \infty$. Set $\omega^N=\omega_1\omega_2\ldots\omega_{n_N}$. By construction, the sequence $\{\FN^{(j)}\}$ is $\left( a^{(j)}, r^{(j)}\right)$-sequence, $j=1,2$.

If $a^{(j)}\neq r^{(j)}$ for $j=1,2$, these sequences are MS and we are done. Otherwise, it is easy to see from the minimality (or even from the absence of finite invariant set) that there exists $\overline{w}$ such that $(\rho^{(j)}(\overline{w}))(a^{(j)}) \neq r^{(j)}$. The application of Lemma~\ref{l:lim-h} to $w^N=\bar{w}\omega^{N}$ concludes the proof.
\end{proof}

As we have said earlier, the key idea of the proof of Theorem~\ref{t:main} is that the translation numbers reflect the cyclic order of ``attractors'' and ``repellers'' on the circle. Namely, let us introduce the following definition:
\begin{Def}
For $f,g\in\HS$, denote by $c(f,g)$ the number
$$c(f,g):=\tau(\tf\circ \tg)-\tau(\tf)-\tau(\tg),$$
where $\tf$ and $\tg$ are respectively the lifts of $f$ and $g$.
\end{Def}
It is easy to see that the value of $c(f,g)$ does not depend on the choice of the lifts~$\tf$ and~$\tg$.

We have the following

\begin{Lm}\label{l:arar}
Let $f,g\in \HS$, and let $\mA_f,\mR_f,\mA_g,\mR_g\subset \bS^1$ be pairwise disjoint intervals such that
$$
f(\bS^1\setminus \mR_f)\subset \mA_f, \quad g(\bS^1\setminus \mR_g)\subset \mA_g.
$$
Then
\begin{itemize}
\item If there is an arc $J\supset \mA_f \cup \mA_g$ that does not intersect neither $\mR_f$ nor $\mR_g$, then $c(f,g)=0$;
\item If the cyclic order of these four arcs is $\mA_f,\mR_f,\mA_g,\mR_g$, then $c(f,g)=1$;
\item If the cyclic order of these four arcs is $\mR_f,\mA_f,\mR_g,\mA_g$, then $c(f,g)=-1$.
\end{itemize}
\end{Lm}

\begin{proof}[Proof of Lemma~\ref{l:arar}]
%\textbf{TODO: change}

Since the maps $f$ and  $g$ have fixed points in the arcs $\mA_f$ and $\mA_g$ respectively, their rotation numbers are zeros. The statement of Lemma \ref{l:arar} does not depend on the particular choice of the lifts $\tf,\tg$, so we can choose the lifts with zero translation number.

%The statement of the lemma does not depend on the choice of lifts $\tf,\tg$, and as the maps $f$, $g$ have fixed points in the arcs $\mA_f$ and $\mA_g$ respectively, we can assume their lifts to be chosen to have zero translation number.
%In not all the four point~$a_f,r_f,a_g,r_g$ are distinct, then the point repeated twice is a common fixed point for $f$ and $g$, and its lift is hence a common fixed point for $\tf$ and $\tg$. Thus, any composition of $\tf$ and $\tg$ has zero translation number.
If $\mA_f$ and $\mA_g$ are adjacent, an arc $J$ that contains them but does not intersect neither $\mR_f$ nor $\mR_g$ is invariant under both~$f$ and~$g$. A lift~$\tJ\subset \mathbb{R}$ of this arc is an interval that is invariant under both~$\tf$ and~$\tg$:
$$\tf(\tJ)\subset \tJ,\ \ \ \tg(\tJ)\subset \tJ.$$
Hence, $\tJ$ is invariant under a composition of $\tf$ and $\tg$, and by Brower theorem this composition has a fixed point in~$\tJ$. Thus, $\tau(\tf\circ\tg)=0$, and hence $c(f,g)=0$.

%\begin{figure}[h]
%\includegraphics[scale=0.7]{rotating-2.pdf}  \qquad \includegraphics[scale=0.7]{rotating-1.pdf}
%\caption{Rotation numbers for the Morse-Smale compositions}\label{f:1}
%\end{figure}

\begin{figure}[h]
\includegraphics[scale=0.9]{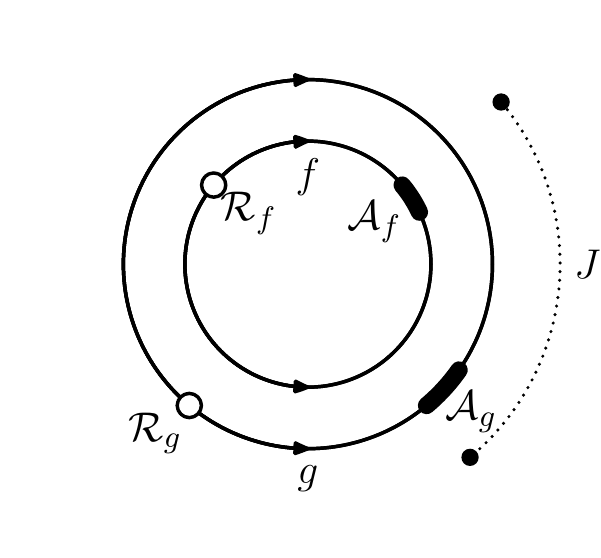}  \qquad \includegraphics[scale=0.9]{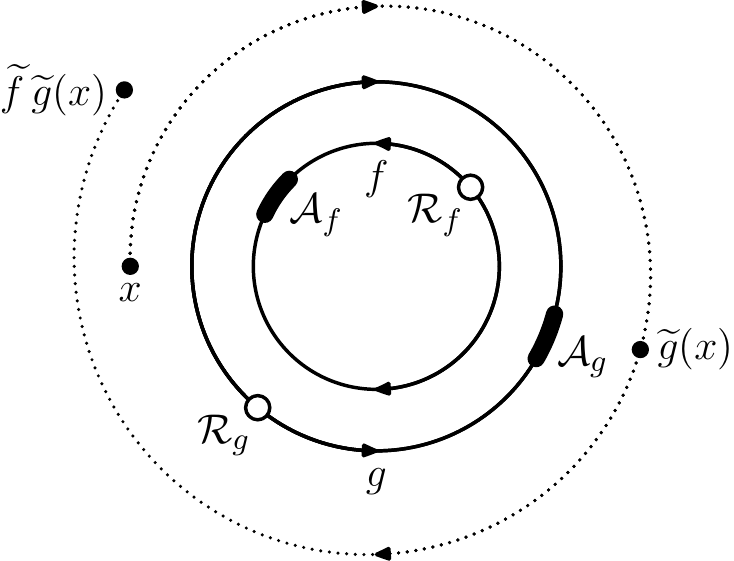}
\caption{Rotation numbers for the Morse-Smale compositions}\label{f:1}
\end{figure}

Assume now that the four arcs are in cyclic order $\mA_f,\mR_f,\mA_g,\mR_g$; choose lifts of these arcs so that
$\tmA_f<\tmR_f<\tmA_g<\tmR_g<\tmA_f+1$, where we write $X<Y$ if $x<y$ for any $x\in X$, $y\in Y$, and both sets are nonempty. Take now a point $x$ between $\tmR_g-1$ and $\tmA_f$. Then $\tg(x)\in \tmA_g$, and hence $\tf \circ\tg(x)\in \tmA_f+1>x+1$; see Fig.~\ref{f:1}. Thus, $\tau(\tf\circ\tg)\ge 1$. On the other hand, taking $x\in\tmA_g$, we have $\tg(x)\in \tmA_g$ and hence
$$
\tf\circ\tg(x)\in\tf(\tmA_g)\subset\tmA_f+1<\tmA_g+1,
$$
and thus obtaining $\tau(\tf\circ\tg)\le 1$. This implies that $\tau(\tf\circ\tg)= 1$. (These arguments are similar to those used by Calegari-Walker in~\cite{Calegari-Walker}.)

The last case is completely analogous.
\end{proof}

Lemma~\ref{l:lim-h} together with with Lemma~\ref{l:arar} immediately imply
\begin{Lm}\label{l:cfgh-h}
Let $\{\FN\}$ be $(a,r)$-MS family, and $g,h\in\HS$ be such that four points $a, r,g(a), h^{-1}(r)$ are pairwise different. Then for any $N$ sufficiently large
$$
c(\FN, g \FN h) = \left\{\begin{array}{ll}
1, & \text{if the points are in the cyclic order }\, a, r,g(a), h^{-1}(r); \\
-1, & \text{if the points are in the cyclic order }\, r, a, h^{-1}(r), g(a);\\
0, & otherwise.
\end{array}\right.
$$
\end{Lm}

%\begin{proof}
%It suffices to take disjoint neighborhoods $\mA$, $\mR$, $\mA'$, $\mR'$ of $a$, $r$, $g(a)$, $h^{-1}(r)$ respectively, and apply Lemma~\ref{l:arar}.
%\end{proof}

%We already see from Lemma~\ref{l:cfgh-h} that the translation numbers ``feel'' the cyclic order of the attractors and repellers of Morse-Smale families and their images.

If Lemma~\ref{l:cfgh-h} was giving us a way of determining the cyclic order of four points completely, we would be able to conclude  by saying that the cyclic order of orbits of ``attractors'' $a^{(j)}$ of the MS-families, given by Proposition~\ref{p:7}, should be the same (as these orders would be determined using only the translation numbers, that coincide for both actions). And hence, we can  map monotonically the points of one of these orbits to the corresponding points of the other one. By definition such a map (extended by continuity as both orbits are dense) would conjugate the two actions.

%attractors and repellers is the same for our two actions $\rho^{(1)}$, $\rho^{(2)}$ as they can be determined via the translation numbers of $\tilde\rho^{(1)}$, $\tilde\rho^{(2)}$, that coincide for the respective words. Hence, we could map the attractors of Morse-Smale maps for one action to the attractors of the corresponding maps of the second action, thus due to Lemma~\ref{l:lim-h} obtaining a conjugacy between the two actions.
However, Lemma~\ref{l:cfgh-h} does not distinguish four different cyclic orders involving an ``attractor-attractor'' arc, and says nothing in the case of coincidence between attractors and repellers.
%What is even more important, we need an argument that will be generalizable to a non-smooth case.
In order to get a fully rigorous argument that handles all these possibilities (even in the non-smooth case!), we will apply the random dynamics argument.

A definition that we need  to handle the case if the assumptions of Lemma~\ref{l:cfgh-h} are not satisfied, is the following one.
\begin{Def}
$$
\Delta(\{\FN\}_{N\in \mathbb{N}},g,h):= \left\{ \begin{array}{ll}
\lim\limits_{N\to\infty} c(\FN, g\FN h), & \text{if the limit exists and equals $\pm 1$},\\
0 & \text{otherwise}.
\end{array}\right.
$$
\end{Def}

Let us again use the random dynamics, corresponding to the chosen measure~$\nu$: we have

\begin{Lm}\label{l:limit}
Let $\{\FN\}$ be an $(a,r)$-MS family, $g(a)\neq a,r$. Then, choosing $h$ to be a \emph{random} length $n$ composition (and denoting the corresponding expectation by $\E_n$), we have
$$
\lim_{n\to\infty} \E_n (\Delta(\{\FN\}_{N\in \mathbb{N}},g,h))= \left\{ \begin{array}{ll}
\mu_{-}([g(a),a]) & \text{if the cyclic order is $a,r,g(a)$};\\
-\mu_{-}([a,g(a)]) & \text{if the cyclic order is $a,g(a),r$},
\end{array}\right.
$$
where $\mu_-=\sum_j \nu(j) (f_j^{-1})_* \mu_-$ is the unique stationary measure for the dynamics of the inverse maps, and $[x,y]$ denotes the positively oriented arc from~$x$ to~$y$.

%\footnote{TODO: Rewrite}
\end{Lm}
%\begin{Cor}
%The point $g(a_f)$
%\end{Cor}

%\

%REWRITE FROM HERE...

%\

\begin{proof}[Proof of Lemma~\ref{l:limit}]
Due to Addendum \ref{p:12}, the distribution of random length $n$ inverse images $h^{-1}(r)$ converges to the stationary measure $\mu_-$ as $n\to\infty$. In particular, the probabilities of all the events $h^{-1}(r)=r$, $h^{-1}(r)=a$ and $h^{-1}(r)=g(a)$ tend to zero as~$n\to\infty$. Hence, with probability that tends to~$1$, assumptions of Lemma~\ref{l:cfgh-h} are satisfied. We thus obtain that in the case if three points $a,r,g(a)$ are in this cyclic order, $\Delta(\FN,g,h)$ is equal to~$1$ if $h^{-1}(r)$ falls on the arc $(g(a),a)$ and zero if it falls on any of two other open arcs. In the same way, in the case of the cyclic order $a,g(a),r$ we see that $\Delta((\FN),g,h)$ equals~$-1$ if $h^{-1}(r)$ falls on the arc $(a,g(a))$ and zero if it falls on any of two other open arcs. The probability of $h^{-1}(r)$ falling on a given arc tends to the $\mu_-$-measure of this arc, and this concludes the proof.
\end{proof}

We are now ready to conclude the proof of the main theorem.  Namely, let $\rho^{(1)}$, $\rho^{(2)}$ be two actions of $F_k^+$ on the circle, such that for their lifts $\tilde\rho^{(1)}$, $\tilde\rho^{(2)}$ for any word $u\in F_k^+$ one has $\tau(\tilde\rho^{(1)})=\tau(\tilde\rho^{(2)})$.

By Proposition~\ref{p:7}, there exist words $w_N$ in $F_k^+$ such that the sequence $\FN^{(j)}:=\rho^{(j)}(w)$ is an $(a^{(j)},r^{(j)})$-family, $j=1,2$. Take an orientation-preserving map $\psi:\bS^1\to\bS^1$ that maps~$a^{(1)}$ to~$a^{(2)}$ and $\mu_-^{(1)}$ to $\mu_-^{(2)}$.

Say that a word $u\in F_k^+$ is \emph{admissible} if for $g^{(1)}:=\rho^{(1)}(u)$ and $g^{(2)}:=\rho^{(2)}(u)$ we have
\begin{equation}\label{eq:diff4}
g^{(1)}(a^{(1)})\neq a^{(1)}, r^{(1)} \quad \text{and } g^{(2)}(a^{(2)})\neq a^{(2)}, r^{(2)}.
\end{equation}
and that a word $u$ is \emph{good} if it is admissible as well as all the words $iu$, for $i=1,\dots,k$.

For any admissible word~$u$, Lemma~\ref{l:limit} implies that the map $\psi$ sends $g^{(1)}(a^{(1)})$ to $g^{(2)}(a^{(2)})$: indeed, the measure~$\mu_-$ of the corresponding intervals is defined in terms of $\Delta(\cdot,\cdot,\cdot)$, that uses only the translation numbers of different compositions and iterations.

Now note, that the probability of a length $n$ word $u$ being admissible tends to~$1$: the measures $\mu_+^{(j)}$ are non-atomic and thus any of the four possible equalities in~\eqref{eq:diff4} has asymptotically vanishing probability.
Hence, the same hold for the probability of $u$ being good. Consider the sets of images $\mathcal{O}^{(j)}$, $j=1,2$ of good images of $a^{(j)}$:
$$
\mathcal{O}^{(j)}:=\{ g^{(j)}(a^{(j)}) \mid g^{(j)}=\rho^{(j)}(u),\, \text{ where } u \text{ is good}\} , \quad j=1,2.
$$
Note that as the probability of a random length $n$ word being good tends to~$1$, these images are asymptotically distributed w.r.t. the stationary measure $\mu_+^{(j)}$, and in particular are dense.

On the other hand, for any good word $u$ the map $\psi$ sends $g^{(1)}(a^{(1)})$ to $g^{(2)}(a^{(2)})$, where $g^{(j)}=\rho^{(j)}(u)$, and $f_i^{(1)}(g^{(1)}(a^{(1)}))=\rho^{(1)}(iu)(a^{(1)})$ to $f_i^{(2)}(g^{(2)}(a^{(2)}))$. Hence,  the equalities $\psi f_i^{(1)}=f_i^{(2)} \psi$, $i=1,\dots, k$ hold on $\mathcal{O}^{(1)}$. This set is dense; by continuity we have $f_i^{(2)}\psi=\psi f_i^{(1)}$ everywhere on the circle, thus obtaining the desired conjugacy between the two actions.

\subsection*{Type (3): factorizable dynamics}

Let $\tilde\rho^{(1)}, \tilde\rho^{(2)}$ be two lifts of actions of type~(\ref{i:cover}) having the same translation numbers. Without loss of generality (making, if necessary, a change of variable), we can assume that the stationary measure for both actions on the circle is the Lebesgue one. Then, the factorization maps $\varphi^{(1),(2)}$ can be taken to be equal to $x\mapsto x+1/l$ (as they should preserve the unique stationary measure). Hence, the lifts of actions of the factorized dynamics are obtained from $\tilde\rho^{(j)}$ ($j=1,2$) by an $l$-times rescaling  (we consider the factor phase space as a circle of length one, not of length~$1/l$). Thus, and their translation numbers are $l$ times bigger than those of $\tilde\rho^{(j)}$, and hence also coincide.

The factorized actions are of type~(\ref{i:generic}), and thus (as we showed above) are topologically conjugate. The conjugacy can be lifted to the initial circle.

This completes the proof of Theorem \ref{t:main}.

\section{Random dynamics}\label{s:random}
For reader's convenience, we recall here the proof of Antonov's Theorem~\ref{t:contraction} from his great work~\cite{A} (that remained mostly unknown to the mathematical community for a long time):
% (with some minor modifications of changing it to the modern language):

%Let us first briefly recall the proof of Antonov's Theorem~\ref{t:contraction}.
\begin{proof}[Proof of Theorem~\ref{t:contraction}]
Take any measure $\mu_-$, stationary for the random dynamics of the maps $f_i^{-1}$, that is, satisfying
$$
\mu_-=\sum_i \nu(i) (f_i^{-1})_* \mu_-.
$$
Then for any $x,y\in S^1$ the random process
$$
\xi_n=\mu_-([F[n,\omega](x),F[n,\omega](y)])
$$
is a martingale, taking values in $[0,1]$. Indeed, it suffices to check that
\begin{multline*}
\E \xi_1 = \sum_i \nu(i) \mu_- ([ f_i(x),f_i(y)]) = \sum_i \nu(i)  \mu_- ((f_i^{-1})^{-1}[x,y]) =
\\ = \sum_i \nu(i) (f_i^{-1})_* \mu_- ([x,y]) = \mu_- ([x,y]) = \xi_0,
\end{multline*}
and then apply it for $x'=F[n,\omega](x), \, y'=F[n,\omega](y)$.

As $\xi_n$ is a bounded martingale, it converges almost surely, i.e. the limit $t(\omega)=\lim_{n\to\infty} \xi_n(\omega)$ exists almost surely (see, e.g.,~\cite[Ch.~VII, \S{}4]{Sh}). Note now that due to standard arguments for any pair of initial points $x,y$ almost surely the limit set $\mL(\omega)$ of the sequence $(F[n,\omega](x),F[n,\omega](y))\subset \T^2$ is forward-invariant under all the maps $f_i\times f_i:\T^2\to \T^2$. On the other hand,  due to the minimality of the dynamics the projection of $\mL(\omega)$ on the first coordinate coincides with the whole circle. As minimality also implies that the measure $\mu_-$ has no atoms and $\text{supp}\, \mu_{-}=S^1$, we have an almost sure description of $\mL(\omega)$:
$$
\mL(\omega)=\{(x',y') \mid \mu_-([x',y'])=t(\omega)\}.
$$
 In other words, a graph of the ``$\mu_-$-rotation'' by $t(\omega)$ almost surely is invariant under all $f_i\times f_i$, and hence such a ``rotation'' commutes with all the~$f_i$.

Now, consider the union of essential images of~$t(\cdot)$ over all the pairs $x,y\in S^1$, and a subgroup of the circle $\R/\Z$ generated by this set. If this subgroup is infinite, by continuity all the maps~$f_i$ commute with \emph{all} the $\mu_-$-rotations, and hence have a common invariant measure~$\mu_-$. We are thus in the case~(\ref{i:measure}). If the subgroup is trivial then for any $x, y\in S^1$ we have $t(\omega)=0$ or~$1$ almost surely, which implies that we are in the case~(\ref{i:generic}). The set $\mL(\omega)$ is the diagonal~$\{x=y\}$ in this case. Finally, if the subgroup is finite, it is generated by a rotation by $1/l$ for some $l\in \mathbb{N}, l>1$, and all the maps $f_i$ commute with a ``$\mu_-$-rotation'' of order $l$. This corresponds to the case~(\ref{i:cover}) in Theorem~\ref{t:contraction}.

Finally, assume that we are in the case~(\ref{i:generic}), and let us prove the convergence of distributions to the stationary measure. Namely, take and fix any stationary measure~$\mu_+$ (that exists due to the standard arguments like time-averaging Krylov-Bogolyubov theorem). In addition to the deterministic initial point $p\in \bS^1$, consider a \emph{random} initial point $p'\in \bS^1$, distributed w.r.t. $\mu_+$, chosen independently from the sequence~$\omega$ of iterations. Then  its random image $F[n,\omega](p')$ is also distributed w.r.t.~$\mu_+$ (by definition of a stationary measure). On the other hand, the distances between random images of $p$ and of $p'$ converge to 0 almost surely. Hence, the distribution of images of $p$ converges to~$\mu_+$.   %\footnote{Do we need it? If no, we should probably write:"It is not hard to show that this convergence is in fact uniform in~$p$, but since we are not using this fact, we do not elaborate on it here." If yes, then some more details are needed.}

Moreover, this implies that the stationary measure is unique (in particular, this implies that the convergence above is uniform in~$p$). Indeed,  if $\mu$ is another stationary measure, let us  take the initial point~$p$ to be regular w.r.t.~$\mu$. On the one hand, the random iterations of~$p$ should stay distributed w.r.t.~$\mu$, on the other hand, their distribution converges to~$\mu_+$, and this leads to a contradiction.
\end{proof}

\begin{proof}[Proof of Addendum~\ref{p:12}]
Theorem~\ref{t:contraction} says that for any two points $x,y\in \bS^1$ their random images approach each other almost surely, and thus almost surely the length of the random the images of the arc $[x,y]$ tends to~$0$ or to $1$:
\begin{equation}\label{eq:arc-im}
\forall x,y\in\bS^1 \quad \lim_{n\to\infty} \left|F[n,\omega]([x,y]) \right| \in \{0,1\} \quad \overline{\nu}-\text{a.s.}
\end{equation}
In particular,~\eqref{eq:arc-im} holds for $\overline{\nu}$-a.e. $\omega$ for all dyadic rational points $x,y$ (as it is a countable family of a.s. satisfied conditions).

Consider any such $\omega$, take any $m\in\N$, and consider $2^m$ points $0,\frac{1}{2^m},\dots,\frac{2^m-1}{2^m}$ and the arcs to which they divide the circle. As the total length of the circle is preserved, there is exactly one of these arcs such that the length of its images tends to~$1$, and the length of images of its complement tends to zero. Denote this (closed) arc by $I_m(\omega)$.

The sequence  $\{I_m\}_{m\in \mathbb{N}}$ form a nested sequence  $I_1\supset I_2\supset \ldots \supset I_m \supset \ldots$, $|I_m|=\frac{1}{2^m}$, and hence has a unique common point $r(\omega)\in \bigcap_m I_m(\omega)$. For any neighborhood~$V$ of this point, for some $m$ one has $I_m\subset V$, and hence
$$
\diam F[n,\omega] (\bS^1\setminus V) \le \diam F[n,\omega] (\bS^1\setminus I_m) \to 0 \quad \text{ as } n\to \infty.
$$

Now, for any $\omega$ there exists at most one point such that the length of images of complement to any of its neighborhoods tends to zero (otherwise, the total length of the image of the circle would tend to zero). Also, % Considering the first iteration we thus see that
for a.e. $\omega'$ we have
$$
r(\sigma(\omega')) = f_{\omega'_1}(r(\omega')),
$$
where $\sigma:\{1, 2, \ldots, k\}^{\mathbb{N}}\to \{1, 2, \ldots, k\}^{\mathbb{N}}$ is the topological Bernoulli shift.
Rephrasing this as $f_i^{-1}(r(\omega))= r(i\omega)$ for all $i=1,\dots,k$ and $\overline{\nu}$-a.e. $\omega$, we see that the diagram
$$
\begin{CD}
\{1,\dots,k\}^\N @>\omega\mapsto i\omega >> \{1,\dots,k\}^\N \\
@VV{r(\cdot)}V @VV{r(\cdot)}V\\
\bS^1 @>f_i^{-1}>> \bS^1
\end{CD}
$$
is $\overline{\nu}$-a.e. commutative for all~$i$. As the measure $\overline{\nu}$ is stationary for the maps $\omega\mapsto i\omega$, its $r(\cdot)$-image~$\mu_-$ is hence stationary for~$f_i^{-1}$.
\end{proof}

\begin{proof}[Proof of Addendum~\ref{p:detector}]

\textbf{Case of type~(\ref{i:measure}) dynamics}: in this case, the translation number is additive under the composition, and the statement is a standard equidistribution ergodic theorem for a system of rotations. Namely, the rotation numbers of compositions that we are considering have the distribution
$$
\mu_N:=\frac{1}{N}\sum_{n=1}^{N} \sum_{i_1+\dots+i_k=n} \left( {n \atop i_1,\dots,i_k} \right) \, \nu(1)^{i_1} \dots \nu(k)^{i_k} \delta_{i_1 \alpha_1+\dots+i_k\alpha_k},
$$
where $\alpha_j=\rho(f_j)$ and their sum is considered on the circle~$\R/\Z$. Indeed, any of $\left( {n \atop i_1,\dots,i_k} \right)$ compositions of length $n$ that has $i_j$ applications of $f_j$ ($j=1,\dots,k$), has the rotation number $i_1\alpha_1+\dots+i_k \alpha_k$ due to the additivity of the rotation number that holds in this case.

Consider the action of $\Z^k$ on the circle $\R/\Z$, with the generators acting by rotations by $\alpha_1,\dots,\alpha_k$ respectively. Then the measure $\mu_N$ is the distribution of image of $0\in\R/\Z$ by a random element of $\Z^k$, chosen w.r.t. the measure
$$
\frac{1}{N}\sum_{n=1}^{N} \sum_{i_1+\dots+i_k=n} \left( {n \atop i_1,\dots,i_k} \right) \, \nu(1)^{i_1} \dots \nu(k)^{i_k} \delta_{(i_1,\dots,i_k)}.
$$
But this sequence of measures (as it is easy to check) is a F\o{}lner one. Hence, the measures obtained using them can accumulate only to common invariant measures of all the rotations~$R_{\alpha_j}$. However, even one irrational rotation has a unique invariant measure, the Lebesgue one. Thus, the measures $\mu_N$ tend to the Lebesgue one.

\textbf{Case of type~(\ref{i:det2}) dynamics}. Informally speaking, Antonov's theorem implies that a sufficiently long random composition $f$ of homeomorphisms $\{f_i\}$ maps a complement to a small arc $\mR$ to a small arc $\mA$. We would like to say that $\mR$ and $\mA$ are almost independent. Indeed, in a sense $\mR$ is mostly determined by the first maps applied, while $\mA$ mostly depends on the last ones. Due to this almost-independence, $\mA$ and $\mR$ are disjoint with probability close to~$1$, thus implying $f(\mA)\subset f(\bS^1\setminus \mR) \subset \mA$. Hence $f$ is likely to have a fixed point in $\mA$, and thus zero rotation number.

%% Replaced \omega -> w

Let us make this approach formal. Namely, set $n':=[n/2]$ and $n''=n-n'$. Decompose the word $w=w_n\dots w_1$ into concatenation $w=w'w''$ of two words of length $n'$ and $n''$ respectively:
$$
F[n,w]=(f_{w_{n'}'}\dots f_{w_{1}'}) \circ (f_{w_{n''}''}\dots f_{w_{1}''})=F_2\circ F_1.
$$
For a given small $\varepsilon>0$ denote by $p_m(\varepsilon)$ the probability that there are two arcs $\mA, \mR\subset S^1$ such that
$$
|\mA|=|\mR|=\varepsilon, \ \ \ \text{and} \ \ \ F[m, \omega](S^1\backslash \mR)\subset \mA.
$$
Due to Antonov's Theorem, for a fixed $\varepsilon>0$ we have $p_m(\varepsilon)\to 1$ as $m\to \infty$.

Let us apply this statement to the maps $F_1, F_2$. Namely, for a given $\varepsilon>0$ with probabilities $p_{n'}(\varepsilon)$ and $p_{n''}(\varepsilon)$, there are arcs
$\mA',\mR'\subset\bS^1$, $\mA'',\mR''\subset\bS^1$  respectively such that
$$
|\mA'|=|\mR'|=\varepsilon, \quad F_1(\bS^1\setminus \mR')\subset \mA', \ \ \text{and}\ \  |\mA''|=|\mR''|=\varepsilon,  \quad F_2(\bS^1\setminus \mR'')\subset \mA''.
$$
Now we would like to exploit the idea that the couples of intervals $(\mA',\mR')$ and $(\mA'',\mR'')$ are defined using different letters of the initial word~$w$, and hence (for any deterministic way of assigning them) are independent. Let us estimate the probability of the event that $\mR'$ and $\mA''$ are defined and $\mR'\cap\mA''=\emptyset$. To do so, let us estimate its the conditional probabilities given $w'$ (and thus $\mR'$, if it is defined).  The arc $\mA''$ is not defined with the probability at most $1-p_{n''}(\eps)$. If it is defined then either $\mR'\cap \mA''=\emptyset$ or $\mA''\subset U_{\eps}(\mR')=:I$. Since $|\mR'|=\eps$, the length of the interval $I$ is equal to $3\eps$.

Now let us use the stationarity of the measure $\mu_+$. It is equal to the average of its images, and hence for any fixed interval  $I$ we have
$$
\mu_+(I) = \E ((F[n'',w''])_*\mu_+) (I) = \E \mu_+(F[n'',w'']^{-1}(I)).
$$
If $w''$ is such that $\mA''$, $\mR''$ are defined and $\mA''\subset I$, then
$$\mu_+ (F[n'',w'']^{-1}(I)) \ge \mu_+ (F[n'',w'']^{-1}(\mA'')) \ge \mu_+(\bS^1\setminus \mR'')\ge 1-\max_{|J|=\eps} \mu_+(J).
$$
Thus, we have
$$
\max_{|J|=3\eps} \mu_+(J) \ge \mu_+(I)\ge \P(\text{$\mA''$ is defined and}\  \mA''\subset I)  \cdot (1-\max_{|J|=\eps} \mu_+(J)).
$$
Therefore
$$
\P(\text{$\mA''$ is defined and}\ \mA''\subset I) \le \frac{\max_{|J|=3\eps} \mu_+(J)}{ (1-\max_{|J|=\eps} \mu_+(J))}. % + (1-p_{n''}(\eps)).
$$
Finally, integrating over $w'$, we have
$$
\P(\text{$\mA''$ and $\mR'$ are defined and}\ \mR'\cap \mA''\neq\emptyset) \le p_{n'}(\eps)\cdot \frac{\max_{|J|=3\eps} \mu_+(J)}{ (1-\max_{|J|=\eps} \mu_+(J))}.
$$
Choosing sufficiently small $\eps$, we can make the right hand side arbitrarily small.  Hence, as $n\to\infty$, the probability of $\mR'$ and $\mA''$ being defined and disjoint tends to~$1$.

Analogously, as $n\to\infty$, the probability of $\mR''$ and $\mA'$ being defined and disjoint tends to~$1$. If $\mR'\cap \mA''=\emptyset$ and $\mR''\cap \mA'=\emptyset$, we have
$$
F(\bS^1\setminus \mR'') =F_1(F_2(\bS^1\setminus \mR'')) \subset F_1 (\mA'') \subset F_1(\bS^1\setminus \mR')\subset \mA' \subset \bS^1\setminus \mR'',
$$
and hence the rotation number of $F$ is zero.

\textbf{Case of type~(\ref{i:det3}) dynamics.} The factorization by $\varphi$ leads to a type~(\ref{i:det2}) dynamics; due to the already proven conclusion, for the factorized dynamics the rotation number of a random length $n$ composition is zero with probability tending to~$1$ as $n\to\infty$. Hence, for the original random dynamics the rotation number with probability tending to~$1$ takes values in~$\{0,\frac{1}{l},\dots, \frac{l-1}{l}\}$. To complete the proof in this case, we have to show that among these values, the rotation number of a random composition of a random length $n$ chosen among $1,\dots,N$ is asymptotically equidistributed as $N\to\infty$.

For simplicity of the following arguments, assume that $\varphi$ is chosen so that for any $x\in\bS^1$ the points $x,\varphi(x),\dots,\varphi^{l-1}(x)$ are in this cyclic order (or, what is the same, the rotation number of $\varphi$ is $1/l$): if this is not the case, we can replace $\varphi$ by its convenient power. We will need the following statement, generalizing one of the conclusions of Antonov's theorem:

\begin{Prop}
In the case of type~(\ref{i:cover}) dynamics, the stationary measure is unique.
\end{Prop}
\begin{proof}
Let $\mu_+$ be the stationary measure for the dynamics on the quotient circle $\bS^1/\varphi$, and $\tilde\mu_+$ be the measure obtained by its lift on the initial circle: for any $I\subset \bS^1$ such that the projection $\pi:\bS^1\to \bS^1/\varphi$ is injective on $I$, we have $\tilde \mu_+(I)=\mu_+(I)$. It is easy to see then that $\frac{1}{l} \tilde \mu_+$ is a stationary measure for the initial dynamics (the normalization constant $\frac{1}{l}$ is needed for this measure to be a probability measure). Let us show that this is the unique stationary measure.

Indeed, any stationary measure for the non-factorized dynamics  is mapped by the projection $\pi$ to a stationary measure for the factorized dynamics, and thus (as the latter is unique) to the measure~$\mu_+$. Now, take any ergodic component $\mu$ of the measure~$\frac{1}{l} \tilde \mu_+$.

Recall that due to the Kakutani Random Ergodic Theorem~\cite{Kakutani} (see also~\cite{Furman}), for any ergodic stationary measure the random trajectory of almost every point is distributed with respect to this measure: for $\mu$-a.e. $x\in\bS^1$ and $\overline{\nu}$-a.e. $\omega\in\{1,\dots,k\}^{\N}$ one has
$
\frac{1}{N} \sum_{n=1}^{N} \delta_{F[n,\omega](x)} \to \mu.
$

Notice now that if measure $\mu$ is ergodic, the same holds for the ``shifted'' measure $\varphi_* \mu$. On the other hand, for any two points $x,y\in \bS^1$ such that the three points $x,y,\varphi(x)$ are in this cyclic order, the probability that the iterations of $y$ and $\varphi(x)$ approach each other is positive. Indeed, the arc $[y,\varphi(x)]$ projects on the arc $\pi([y,\varphi(x)])$ on the quotient circle that does not coincide with all the circle. Hence the probability that $r(\omega)\not\in \pi([y,\varphi(x)])$ is positive, thus implying positive probability that the diameter of iterations of the projected arc tends to zero. Hence, the same holds for the diameter of the iterations of the initial arc.

If $\mu$ is an ergodic stationary measure, it is non-atomic as it projects to the non-atomic stationary measure $\mu_+$ on the circle.  There are two intervals $I_1,I_2$ of positive $\mu$-measure on the circle, such that $I_2$ lies between $I_1$ and $\varphi(I_1)$. Indeed, just take any small interval of positive $\mu$-measure and divide it into two subintervals of positive measure.  For a $\mu$-generic $x\in I_1$, $y\in I_2$ for $\overline{\nu}$-a.e. $\omega$ we have
$$
\frac{1}{N}\sum_{n=1}^N \delta_{F[n,\omega](y)} \rightarrow \mu,
$$
$$
\frac{1}{N}\sum_{n=1}^N \delta_{F[n,\omega](\varphi(x))} \rightarrow \varphi_*\mu,
$$
and with positive probability the iterations of $\varphi(x)$ and $y$ approach each other. Hence, $\mu=\varphi_*\mu$. Thus, the only ergodic stationary measure for the initial system is the $\varphi$-invariant measure~$\frac{1}{l} \tilde \mu_+$. As any stationary measure can be decomposed into ergodic components, it is the only stationary measure of the system.
\end{proof}

This unique ergodicity by the standard argument implies that the time-averages of distributions of random images of any point $p\in\bS^1$ converge to $\frac{1}{l} \tilde \mu_+$, and the convergence is uniform in $p$ (otherwise we would be able to extract a subsequence that converges to a stationary measure different from $\frac{1}{l} \tilde \mu_+$).

Now let us fix a large arbitrary $m\in\N$. Randomly chosen length $n$ between $1$ and $N>>m$ of the composition can be decomposed as $n=n'+n''$, where $n''=[n/m]\cdot m$, $n'=n-n''$ is the remainder of division of $n$ by~$m$. If $N$ and $m$ are large and $N$ is much larger than $m$ then the probability that $n''$ takes maximal possible value, i.e. $n''=[N/m]\cdot m$, is small. Then conditionally to any value of $n''$ other than the maximal one, the remainder $n'$ is uniformly distributed among $0,1,\dots,m-1$.

Let us lift the construction used in the proof for the type~(\ref{i:generic}) situation to our $l$-leaves cover. Namely, we have seen in the proof for that case that decomposing $n=n'+n''$ and considering separately the maps associated to subwords $w'$ and $w''$ formed by first $n'$ and last $n''$ letters respectively, with large probability we can find independent couples of intervals $(\mA',\mR')$ and $(\mA'',\mR'')$. With probability that tends to~$1$ as $m$ and $N$ tend to infinity these intervals are small and satisfy the disjointness conditions $\mR'\cap\mA''=\emptyset$, $\mR''\cap\mA'=\emptyset$.

Lifting this to our cover, we see that $(\mA',\mR')$ becomes a $2l$-tuple of intervals $(\mA'_1,\dots,\mA'_l,\mR'_1,\dots,\mR'_l)$, where $\mA'_j=\varphi^{j-1}(\mA'_1)$, $\mR'_j=\varphi^{j-1}(\mR'_1)$, similarly for $\mA''_j, \mR''_j, \quad j=1,\dots, l$. Denote also by $I_j$ the connected components of complement to the union of $\mR_j''$, numbering them so that $I_j$ lies between $\mR_j''$ and $\mR_{j+1}''$, and choose the beginning of numbering of $\{\mA_j''\}$ so that $F[n'',w''](I_j'')\subset\mA_j''$.

We know that with probability close to~$1$ the disjointness condition is satisfied. In this case
$$
F[n',w'](F[n'',w''](I_1))\subset F[n',w'](\mA_1'')\subset \bigcup_j \mA_j'\subset \bigcup_j I_j,
$$
and the rotation number of the composition $F[n,w]=F[n',w']\circ F[n'',w'']$ is defined by the index $r$ such that $F[n,w](I_1)\subset I_r$, namely, it is equal to~$\frac{r-1}{l}$.

For any $n''<[N/m]\cdot m$ (the probability of this tends to~$1$ as $N\to\infty$), and any word $w''$ of length $n''$, fix a point $p\in \mA_1''$ (provided that~$\mA_1''$ is defined). Note that its image $F[n',w'](p)$ will belong to $I_r$ (provided that the disjointness condition is satisfied). On the other hand, distribution of such random images of random length $n'$, equidistributed among $\{0,\dots,m-1\}$, is $\frac{1}{m}\sum_{n'=0}^{m-1} T_{\nu}^{n'} \delta_p$. As we've discussed earlier, it converges as $m\to\infty$ to the $\varphi$-invariant stationary measure~$\frac{1}{l} \tilde \mu_+$.

In particular, as $m\to\infty$, the events $F[n',w'](p)\in I_j=\varphi^{j-1}(I_1)$ become more and more equiprobable. Averaging this among the choices of~$w''$, we see that as $N\to\infty$, the rotation number of $F[n,w]$ tends to be equidistributed on $\{0,\frac{1}{l},\dots, \frac{l-1}{l}\}$.

\end{proof}

\end{document}